\documentclass[11pt, twoside]{article}
\usepackage{amsfonts}
\usepackage{amsmath,amssymb}
\usepackage{multicol}
\usepackage{color}

\newtheorem{theorem}{Theorem}[section]
\newtheorem{lemma}{Lemma}[section]

\newtheorem{definition}{Definition}[section]
\newtheorem{remark}{Remark}[section]

\numberwithin{equation}{section}
\newenvironment{proof}[1][Proof]{\noindent\textbf{#1.} }{\hfill $\Box$}
\allowdisplaybreaks

 \makeatletter
\begin{document}
\author{Long-Jiang Gu$^a$ and  Huan-Song Zhou$^{b}$\thanks{Corresponding
author: hszhou@whut.edu.cn}
 \\
 \\
{\small   $^a$Wuhan Institute of Physics and Mathematics, Chinese Academy of Sciences,}\\
{\small P.O. Box 71010, Wuhan 430071, China}\\
{\small $^b$Department of Mathematics, Wuhan University of Technology, Wuhan 430070, China}\\
}
\title{ An improved fountain theorem   and its application
  }
 \date{}
  \maketitle

\begin{abstract}
The main aim of the paper is to prove a   fountain theorem without assuming the $\tau$-upper semi-continuity condition on the variational functional. Using this improved fountain
theorem, we may deal with more general strongly indefinite elliptic problems with various sign-changing nonlinear terms.  As an application,
 we obtain infinitely many solutions for a semilinear Schr$\ddot{\text{o}}$dinger equation with strongly indefinite
 structure and sign-changing nonlinearity.
\end{abstract}

\section{Introduction}
\noindent

Since the pioneering works of Bartsch and Willem \cite{old fountain,He} (see also, \cite{MW}), variant fountain theorems are established and which have been  used to study the existence
of infinitely many solutions for various elliptic problems, see e.g., \cite{old fountain,He,Bartsch3,BC1,BC2,hz,ZouWM} and the references  therein. In order to investigate
infinitely many critical points of strongly indefinite functionals, Batkam-Colin \cite{BC1} established a generalized fountain theorem based on the so-called
$\tau$-topology introduced by Kryszewski-Szulkin \cite{SK}. For recalling the fountain theorem proved in \cite{BC1}, we introduce some notations and definitions
which are also often used in the following sections of the paper.

Let $X$  be a separable Hilbert space and $Y\subset X$ be a closed subspace of $X$
endowed with  inner product $\langle\cdot,\cdot\rangle$ and   norm $\| \cdot \|$. Let
\begin{equation}
    X=Y\bigoplus Z \text{ with } Z=Y^{\bot}, Y=\overline{\bigoplus_{j=0}^{\infty}\mathbb{R}e_{j}}  \text{ and }
    Z =\overline{\bigoplus_{j=0}^{\infty}\mathbb{R}f_{j}},
\end{equation}
where $\{e_{j} \}_{j\geq0}$ and $\{f_{j} \}_{j\geq0}$ are orthonormal  bases of $Y$ and $Z$, respectively. Moreover, we define

\begin{equation}\label{fountain.1.1}
   Y_{k}:=Y \bigoplus (\bigoplus_{j=0}^{k}\mathbb{R}f_{j})
 \quad\quad \text{and} \quad\quad
 Z_{k}:=\overline{\bigoplus_{j=k}^{\infty}\mathbb{R}f_{j}},
\end{equation}
and let
\begin{equation}
    P : X \rightarrow Y, \quad   Q : X \rightarrow Z \quad  \text{and}\quad   P_k : X \rightarrow Y_{k-1}, \quad   Q_k : X \rightarrow Z_k
\end{equation}
be the orthogonal projections. The $\tau$-topology on $X=Y\bigoplus Z$ introduced in \cite{SK} is the topology associated to the following
norm
\begin{equation}\label{fountain.tau norm}
     \| u \|_{\tau}:=
      max\{\sum_{j=0}^{\infty}\frac{1}{2^{j+1}}|\langle Pu,e_{j}\rangle|,
      \| Qu \|\}  ,   \text{ for } u\in X.
\end{equation}
By the above definition, we see that
\begin{equation}\label{fountain.norminequality}
    \|u\|_{\tau}\leq max\{ \| Pu \|  , \| Qu \| \}   \leq\|u\| \text{ for } u\in X.
\end{equation}
Furthermore, it follows from \cite{SK} and the appendix of \cite{weak linking} that, if   $\{u_{n}\}\subset X$ is   bounded, then
\begin{equation}\label{fountain.1.remark1}
    u_{n}\overset{n}{\rightarrow} u \text{ in } \tau\text{-topology} \Leftrightarrow
  Pu_{n}\overset{n}{\rightharpoonup} Pu   \text{ and }   Qu_{n}\overset{n}{\rightarrow} Qu.
\end{equation}

\begin{remark}\label{fountain.remark2}
 Let
$$  \tilde{e}_{j}=
 \left\{
   \begin{array}{ll}
     f_{j},  \quad j=0,1,...,k-1\\
     e_{j-k},   \quad j\geq k ,
   \end{array}
   \right.$$
and define the following norm
\begin{equation}\label{fountain.tau k norm}
    \|u\|_{\tau_{k}}=max\{\sum_{j=0}^{\infty}\frac{1}{2^{j+1}}|\langle P_{k}u,\tilde{e}_{j} \rangle|,
      \| Q_{k}u \|\},
\end{equation}
then, $\| \cdot \|_{\tau_{k}}$ and $\| \cdot \|_{\tau}$ are equivalent for all $k\geq1$.
In fact, it is enough  to show that $\| \cdot \|_{\tau_{1}}$ and $\| \cdot \|_{\tau}$ are equivalent. By (\ref{fountain.tau norm}),
\begin{eqnarray*}
  \| u \|_{\tau}^2 &\leq&  \Big(\sum_{j=0}^{\infty}\frac{1}{2^{j+1}}|\langle  u,e_{j}\rangle|\Big)^2+ \sum_{j=0}^{\infty}|\langle u,f_{j}\rangle|^2 \\
    &=&  \Big(\sum_{j=0}^{\infty}\frac{1}{2^{j+1}}|\langle  u,e_{j}\rangle|\Big)^2+|\langle u,f_{0}\rangle|^2+ \sum_{j=1}^{\infty}|\langle u,f_{j}\rangle|^2 \\
    &\leq&   4\Big(\frac{1}{2}|\langle u,f_{0}\rangle|+\sum_{j=1}^{\infty}\frac{1}{2^{j+1}}|\langle  u,e_{j-1}\rangle|\Big)^2 +4 \sum_{j=1}^{\infty}|\langle u,f_{j}\rangle|^2
    \leq   8\| u \|_{\tau_1}^2.
\end{eqnarray*}
On the other hand,
\begin{eqnarray*}
  \| u \|_{\tau_1} &\leq&    \frac{1}{2}|\langle u,f_{0}\rangle|+\sum_{j=0}^{\infty}\frac{1}{2^{j+2}}|\langle  u,e_{j}\rangle|  + \Big( \sum_{j=1}^{\infty}|\langle u,f_{j}\rangle|^2 \Big)^{\frac{1}{2}} \\
    &\leq&  \sum_{j=0}^{\infty}\frac{1}{2^{j+1}}|\langle  u,e_{j}\rangle|+2^{\frac{1}{2}}\Big(|\langle u,f_{0}\rangle|^2+ \sum_{j=1}^{\infty}|\langle u,f_{j}\rangle|^2 \Big)^{\frac{1}{2}}
    \leq   2^{\frac{3}{2}}   \| u \|_{\tau }.
\end{eqnarray*}

\end{remark}

Throughout the paper, for  $r_{k}>0$ and $\rho_{k}>0$, we always set
\begin{equation}\label{fountain.1.2}
    B_{k}:=\{u\in Y_{k} : \|u\| <\rho_{k}\}
\quad and \quad
N_{k}:=\{u\in Z_{k} : \|u\| =r_{k}\}.
\end{equation}

Since $X$ is a Hilbert space, if $\varphi \in C^{1}(X,\mathbb{R})$, $\nabla\varphi$ is given by the formula
$$  \langle \nabla\varphi(u),v \rangle= \varphi'(u)v, \text{ for all } v \in X.   $$

With the above notations, for an even functional the fountain theorem proved in \cite{BC1} can be stated as follows


\begin{theorem}\cite[Corollary 13]{BC1}\label{fountain.Batkam's Fountain Thm}
Let $\varphi \in C^{1}(X,\mathbb{R})$ be  an even functional satisfying
 \begin{itemize}
   \item  $\nabla\varphi$ is weakly sequentially continuous, i.e., for any $v \in X$, $\varphi'(u_n)v \overset{n}{\rightarrow} \varphi'(u)v$ if
   $u_n \overset{n}{\rightharpoonup} u$ weakly in $X$.
   \item  $\varphi$ is  $\tau$-upper semi-continuous, i.e., for any $c \in\mathbb{R}$, the set
   $\varphi_{c}:=\{u\in X: \varphi(u)\geq c\}$
is $\tau$-closed.
   \item For any $c>0$, $\varphi$ satisfies $ (PS)_c$ condition, i.e., any
sequence $\{u_{n}\}\subset X$ with
 $$  \varphi(u_{n})\overset{n}{\rightarrow} c \quad  \text{and}  \quad \varphi'(u_{n})\overset{n}{\rightarrow} 0 \text{ in } X' \ (\text{the dual space of }X)$$
has a convergent subsequence.
 \end{itemize}
Additionally, if there exist   $\rho_{k}>r_{k}>0$ such that:
\[ d_{k}:=\sup\limits_{u\in Y_{k},\|u\|  \leq \rho_{k}} \varphi(u) < \infty , \leqno{(A_{1})}\]
 \[ a_{k}:=\sup\limits_{u\in Y_{k},\|u\|= \rho_{k}}\varphi(u)\leq 0,\leqno{(A_{2})}
  \]
 \[b_{k}:=\inf\limits_{u\in Z_{k},\|u\|= r_{k}}\varphi(u) \rightarrow \infty \text{ as }  k \rightarrow \infty. \leqno{(A_{3})}
 \]
\\Then, $\varphi$ has  an unbounded  sequence of critical values.
\end{theorem}

 The $\tau$-upper semi-continuity was proposed in \cite{SK} for showing  a generalized linking theorem. Similar to \cite{SK},
 this condition is also required in the above Theorem \ref{fountain.Batkam's Fountain Thm}, which is mainly used to construct a suitable vector field.
Theorem 1.1 can be used to deal with some strongly indefinite elliptic problems, but the $\tau$-upper semi-continuity assumption
requires that  the primitive functions of the nonlinearities of the elliptic problems should be positive, see condition $(\mathbf{f_4})$
in \cite{BC1}. It is natural to ask what would happen if the nonlinear terms of an elliptic problem change  sign and lose the positivity
condition $(\mathbf{f_4})$? So, the main aim of the paper is to establish a variant fountain theorem without assuming the $\tau$-upper semi-continuity
and then we may answer the above question, see our Theorems \ref{fountain.Fountain Thm} and \ref{fountain.existence result}. Our proofs are motivated by the papers \cite{BC1,CH,Rup}. We mention that if the
$\tau$-upper semi-continuity of $\varphi$ is removed, several steps in the proofs for Theorem 1.1 in \cite{BC1} seem not working any more,
for examples:
\begin{itemize}
  \item We cannot construct the pseudogradient vector by the same way as in \cite{BC1} since the set $\varphi^{-1}(-\infty,c)$ may
  not be $\tau$-open now. This difficulty is overcome in this paper by using some ideas from \cite{CH}.
  \item To the authors' knowledge, the intersection lemma used in \cite{BC1} is no longer applicable since the descending flow in our paper
  has different behavior from that of \cite{BC1}. In this paper, we use the intersection lemma given in \cite{Rup} instead.
  \item We cannot make an explicit mini-max characterization on the critical values of $\varphi$ because of the lack of $\tau$-upper semi-continuity
  for $\varphi$, then it is hopeless to get infinitely many different critical points of $\varphi$ by comparing their critical values as that
  of \cite{BC1} or \cite{old fountain}. In this paper, we get infinitely many different critical points $\{u_n\}$ of $\varphi$ by  comparing their norm $\|u_n\|$ and
  proving $\|u_n\|\overset{n}{\rightarrow} +\infty$.
\end{itemize}

Now,  we give our improved fountain theorem:

\begin{theorem}\label{fountain.Fountain Thm}
Let $\varphi \in C^{1}(X,\mathbb{R})$ be an even functional
satisfying $ (PS)^c$ condition (i.e., any
sequence $\{u_{n}\}\subset X$ with $ \sup\limits_{n} \varphi(u_{n}) \leq c$  and $\varphi'(u_{n})\overset{n}{\rightarrow} 0$
having a convergent subsequence)
 and let $\nabla\varphi$ be weakly sequentially continuous.
For any $k \in \mathbf{N}$, if there exists $\rho_{k}>r_{k}>0$ such that, in addition to
the above assumptions $(A_1)$ 
and $(A_3)$,
there holds 
\[
a_{k}:=\sup\limits_{u\in Y_{k},\|u\|= \rho_{k}}\varphi(u)< \inf\limits_{u\in Z_{k},\|u\|\leq r_{k}}\varphi(u), \leqno{(A_{2})'}
\]
\[
\sup\limits_{\|u\|_{\tau}<\delta }\varphi(u) \leq C_{\delta} < \infty, \text{ for any } \delta>0,\leqno{(A_{4})}
\]
then, $\varphi$ has  a sequence of critical points $\{u^{k_m}\}$ such that
    $\lim\limits_{m\rightarrow\infty}\|u^{k_m}\| \rightarrow  \infty$.

\end{theorem}

\begin{remark}\label{rem1.2}
In our Theorem \ref{fountain.Fountain Thm}, the $\tau$-upper semi-continuity is not assumed. But, we
 replaced condition $(A_{2})$ in Theorem \ref{fountain.Batkam's Fountain Thm} by $(A_{2})'$
 and added a new assumption $(A_{4})$. However, the conditions $(A_{2})'$ and $(A_{4})$ are easily verified in the applications, see e.g., the proof of our Theorem \ref{fountain.existence result}.
\end{remark}

With the above Theorem \ref{fountain.Fountain Thm}, we may study the existence of infinitely many solutions for the
 following Schr$\ddot{\text{o}}$dinger equation with strongly indefinite linear part
 and sign-changing nonlinear term:

\begin{equation}\label{fountain.P}
      \left\{
   \begin{array}{ll}
     -\Delta u + V(x)u=g(x)| u|^{q-2}u+h(x)| u|^{p-2}u,\\
     u\in H^{1}(\mathbb{R}^{N}) , \quad N\geq3,
   \end{array}
   \right.
\end{equation}
where
 $1<q<\frac{p}{p-1}<2<p<2^{\ast}$  and  $2^{\ast}=\frac{2N}{N-2}$, $V(x)$, $g(x)$ and $h(x)$ are functions satisfying

\begin{description}
  \item[$(H_{1})$]  $V(x) \in C(\mathbb{R}^N,\mathbb{R})\bigcap L^{\infty}$  and $0$ lies in a spectrum gap of the operator $-\Delta+V$.
  \item[$(H_{2})$] $g\in L^{q_{0}}(\mathbb{R}^{N})\bigcap L^{\infty}(\mathbb{R}^{N})$ with $q_{0}=\frac{2N}{2N-qN+2q}$.
  \item[$(H_{3})$] $h\in L^{p_{0}}(\mathbb{R}^{N})\bigcap L^{\infty}(\mathbb{R}^{N})$ with $p_{0}=\frac{2N}{2N-pN+2p}$ and $h(x)>0$ a.e. in $\mathbb{R}^{N}$.
\end{description}

Since $0$ lies in a gap of the spectrum of $-\Delta+V$, problem (\ref{fountain.P}) may be
strongly indefinite. The nonlinearity in (\ref{fountain.P}) has a   super-linear part and a
sub-linear part which is usually called  concave-convex nonlinearity,   some well-known results corresponding to concave-convex nonlinearities can be found in \cite{G,He} and the references therein. By   $(H_{2})$, we   see that the weight function $g(x)$ may change sign, so, the variational functional of  (\ref{fountain.P})
does not satisfy the $\tau$-upper semi-continuous assumption.

We mention that there are some papers  on the existence of solutions for Schr$\ddot{\text{o}}$dinger equations with  both sign-changing
potential $V(x)$ and indefinite nonlinearities, see, e.g., \cite{cos,cos3,cos2,J F Yang,F Z Wang,wx}, etc.
 But the problems discussed in  \cite{cos,cos3,cos2,F Z Wang,wx} do not have strongly indefinite structure, and in paper \cite{J F Yang}
 the potential $V(x)$ has to be periodic and the weight functions in nonlinear term   must satisfy some additional  conditions.
 There seems no results for problem   (\ref{fountain.P})  under the conditions $(H_1)$-$(H_3)$.
For problem (\ref{fountain.P}), we have the following theorem.
\begin{theorem}\label{fountain.existence result}
If the conditions $(H_{1})$-$(H_{3})$ hold, then
problem (\ref{fountain.P}) has a sequence of nontrival solutions $\{u_{n}\}\subset H^{1}(\mathbb{R}^N)$ with
$\|u_{n}\|_{H^1} \rightarrow  \infty$, as $n\rightarrow\infty$.
\end{theorem}

\section{Proof of our fountain  theorem}
\indent

In this section, we are going to  prove our fountain  theorem, that is,  Theorem \ref{fountain.Fountain Thm}.
For doing this,    some    lemmas are required.

\begin{lemma}\label{fountain.guodu}
Let $ Y_k$ and $Z_k $ be defined in  (\ref{fountain.1.1}). $\varphi \in C^{1}(X,\mathbb{R})$ is an even functional and $\nabla\varphi$ is weakly sequentially continuous.
If there exist   $k\in \mathbf{N}$   and  $\rho_{k}>r_{k}>0$ such that
conditions $(A_1)$   $(A_2)'$ are satisfied
and there holds
\begin{equation}\label{fountain.jihe.1}
    b_{k}:=\inf\limits_{u\in Z_{k},\|u\|= r_{k}}\varphi(u)>\sup\limits_{\|u\|_{\tau}<\delta  }\varphi(u), \text{ for some  } \delta>0.
\end{equation}
Then, there exists   a sequence   $\{u^k_{n}\}\subset \varphi^{d_{k}+1}:=\{u \in X: \varphi(u)\leq{d_{k}+1}\}$  such that
$$\inf\limits_{n}\|u^k_{n}\|_{\tau}\geq \frac{\delta }{2} \text{ and }
\varphi'(u^k_{n})\overset{n}{\rightarrow} 0  \text{ in } X' (\text{the dual space of } X).$$

\end{lemma}

In order to prove Lemma \ref{fountain.guodu}, we need the following deformation lemma:

\begin{lemma}\label{fountain.Deformation Lemma}
Under the assumptions of Lemma \ref{fountain.guodu}, let
\begin{equation}\label{fountain.definition of E}
    E=\varphi^{d_k+1} \bigcap\{u\in X:\|u\|_{\tau}\geq\frac{\delta }{2}\},
\end{equation}
 with   $\varphi^{d_k+1}:=\{u\in X:\varphi(u)\leq d_{k}+1\}$  and $\delta$ given in (\ref{fountain.jihe.1}),
if there exists $\epsilon \in (0,\frac{1}{2})$  with
\begin{equation}\label{fountain.epsilon}
    0<\epsilon<b_{k}-max\{a_{k}, \sup\limits_{\|u\|_{\tau}<\delta  }\varphi(u)\},
\end{equation}
such that
$$\|\varphi'(u)\|>\epsilon, \text{ for any } u\in E,$$
then, there exist $T>0$ and a map $\eta(t,u) \in C([0,T] \times B_{k},X)$ with $B_k$ given by (\ref{fountain.1.2}) such that
\begin{description}
  \item[(i)] $\eta(0,u)=u$ and $\eta(t,-u)=-\eta(t,u)$ for any $u\in B_{k}$ and $t \in [0,T]$.
  \item[(ii)] $\varphi(\eta(t,u))$ is non-increasing in $t\in [0,T]$ for fixed $u\in B_{k}$.
  \item[(iii)] $\eta$ is $\tau$-continuous (i.e.,
              $\eta(t_{m},u_{m}) \overset{m}{\rightarrow} \eta(t,u)$ in   $\tau$-topology,
           if  $t_{m}\overset{m}{\rightarrow} t$  and $u_{m}\overset{m}{\rightarrow} u$  in  $\tau$-topology)
          and $\eta(t,\cdot):B_{k}\rightarrow \eta (t, B_{k})$ is a $\tau$-homeomorphism  for any $t\in [0,T]$.
 \item[(iv)] $\eta(T,B_{k})\subset \varphi^{b_{k}-\epsilon}$.
  \item[(v)] For any $(t,u)\in [0,T] \times B_{k}$, there exist a neighborhood $W_{(t,u)}$ of $(t,u)$ in
          the $|\cdot|\times\tau-$ topology such that
          $$\{  v-\eta(s,v)| (s,v)\in  W_{(t,u)}\bigcap  ([0,T] \times B_{k}  )       \}$$
          is contained in a finite-dimensional subspace of $X$.

\end{description}

\end{lemma}

\begin{proof}
For $\epsilon>0$ given in (\ref{fountain.epsilon}), let
\begin{equation}\label{fountain.R}
    B_{R}=\{u \in X: \|u\|\leq R\}, \text{ where } R= 2(d_k-b_k +2 \epsilon) / \epsilon  +\rho_k +\delta.
\end{equation}
Firstly, we  claim that there exists  a   vector field  $\chi: \varphi^{d_k+1}\rightarrow X$ such that
\begin{description}
  \item[(a)] $\chi$ is odd with $\|\chi(u)\|\leq 2/\epsilon$ and $\langle \nabla \varphi(u),\chi(u)\rangle \leq 0$,   for any   $u\in \varphi^{ d_k+\epsilon}$.
   \item[(b)]   $\chi(u)$ is   locally Lipschitz continuous and $\tau$-locally Lipschitz $\tau$-continuous   on $\varphi^{ d_k+\epsilon}$.
  \item[(c)] $\langle \nabla\varphi(u),\chi(u)\rangle <-1$,   for any  $ u\in \varphi^{-1}[b_{k}-\epsilon,d_k+\epsilon)\bigcap B_{R}$.
  \item[(d)] For any $u\in \mathcal{W}$, $\mathcal{W}$ is given by (\ref{fountain.2.huaW}), there exist a $\tau$-open neighborhood $U_{u} \in \mathcal{N}$
of $u$ such that   $\chi(U_{u})$ is contained in a finite-dimensional subspace of $X$.
\end{description}

In fact, by our  assumption,  $\|\varphi'(u)\|>\epsilon$ for any $u\in E$,
we  may define
$$\omega(u)=\frac{2 \nabla\varphi(u)}{\|\nabla\varphi(u)\|^{2}}, \text{ for }  u\in E\bigcap B_{R},$$
then, there exists a $\tau$-neighborhood $V_{u}\subset X$ of $u$ such that
\begin{equation}\label{fountain.2.xingxing}
    \langle \nabla\varphi(v),  \omega(u) \rangle >1 , \quad \text{for any} \quad v\in V_{u}\bigcap B_{R}.
\end{equation}
Otherwise, if such $V_{u}$ does not exist, then there exists a sequence $\{v_n\}\subset B_{R}$
such that  $v_n \overset{\tau}{\rightarrow} u $ and $\lim\limits_{n\rightarrow\infty}\langle \nabla\varphi(v_n),  \omega(u)\rangle\leq 1$.
By  (\ref{fountain.1.remark1}) we have $v_n\rightharpoonup u$ weakly in $X$ and this leads to a contradiction
since $\nabla\varphi $ is weakly continuous and $\langle \nabla\varphi(u),  \omega(u)\rangle=2$.

Note that $B_{R}$ is $\tau$-closed \cite{CH}, thus $X \backslash B_{R}$ is $\tau$-open, and
\begin{equation}\label{fountain.2.huaN}
    \mathcal{N}=\{V_{u}:u\in E\bigcap B_{R}\} \bigcup \{ X \backslash B_{R}\}
\end{equation}
forms a $\tau$-open covering of $E$.

Since $\mathcal{N}$ is metric, hence paracompact, there exists a local finite $\tau$-open covering
$\mathcal{M}=\{M_{i}:i\in \Lambda \}$, where $\Lambda$ is an index set, of $E$ finer than $\mathcal{N}$. If $M_{i}\subset V_{u_{i}}$
for some $u_{i}\in E$, we choose $\omega_{i}=\omega(u_{i})$ and if $M_{i}\subset X\backslash B_{R}$,
we choose $\omega_{i}=0$. Let $\{\lambda_{i}(u):i \in \Lambda\}$ be a $\tau$-Lipschitz continuous partition
of unity subordinated to $\mathcal{M}$ and let
$$\xi(u)=\sum_{i \in \Lambda}\lambda_{i}(u)\omega_{i}, \quad u\in \mathcal{N}.$$
Since the $\tau$-open covering $\mathcal{M}$ of $\mathcal{N}$ is local finite, each $u\in\mathcal{N}$ belongs
to finite many sets $M_{i}$. Therefore, for every $u\in \mathcal{N}$, the sum $\xi(u)$ is only a finite sum.
It follows that, for any $u\in \mathcal{N}$, there exist a $\tau$-open neighborhood $U_{u} \in \mathcal{N}$
of $u$ such that  $\xi(U_{u})$ is contained in a finite-dimensional subspace of $X$. Then, by the equivalence of  norms   in a finite-dimensional vector space,
we know that there exists $C>0$ such that
\begin{equation}\label{fountain.2.gai1}
     \|\xi(v)-\xi(w)\|\leq    C\|\xi(v)-\xi(w)\|_{\tau}, \quad \forall v,w \in U_{u}.
\end{equation}
On the other hand, by the $\tau$-Lipschitz continuity of $\lambda_i$ and (\ref{fountain.norminequality}), we have that
there exists a constant $L_{u}>0$ such that
\begin{equation}\label{fountain.2.gai2}
    \|\xi(v)-\xi(w)\|_{\tau}\leq L_{u} \|v-w\|_{\tau} \leq L_{u} \|v-w\|, \quad \forall v,w \in U_{u}.
\end{equation}
   Then, from (\ref{fountain.2.gai1}) and (\ref{fountain.2.gai2})
we  know that $\xi(u)$ is locally
Lipschitz continuous and $\tau$-locally Lipschitz $\tau$-continuous. Moreover, by (\ref{fountain.2.xingxing}) and the property of  $\lambda_i$, we also have that
$$\langle \nabla\varphi (u),\xi(u)\rangle >1  \text{ and } \|\xi(u)\|<\frac{2}{\epsilon},
\text{ for any } u\in E\bigcap B_{R}.$$

Since $\varphi$ is even, $\mathcal{N}$ is symmetric, we  define $ \tilde{\xi}(u):=\frac{\xi(u)-\xi(-u)}{2}$ for $u \in \mathcal{N}$, and $\tilde{\xi}(u)$ is odd.
For $\delta>0$ given by (\ref{fountain.jihe.1}), let $\theta \in C^{\infty}(\mathbb{R},[0,1])$ such that

 $$  \theta(t)=
 \left\{
   \begin{array}{ll}
     0,  \quad 0\leq t\leq \frac{2\delta }{3},\\
     1,   \quad t\geq \delta .
   \end{array}
   \right.$$
Define the vector field $\chi: \mathcal{N}\rightarrow X$ by

\begin{equation}\label{fountain.chi}
     \chi(u)=
 \left\{
   \begin{array}{ll}
     -\theta(\|u\|_{\tau}) \tilde{\xi}(u),  \quad u\in \mathcal{N},\\
     0,   \quad \|u\|_{\tau}\leq\frac{2}{3}\delta .
   \end{array}
   \right.
\end{equation}
It's easy to see that $\chi$ is an odd vector field and also well defined on
\begin{equation}\label{fountain.2.huaW}
    \mathcal{W}:=\mathcal{N}\bigcup\{u\in X: \|u\|_{\tau}<\delta \},
\end{equation}
satisfying
\begin{equation}\label{fountain.2.xing}
    \|\chi(u)\|\leq \frac{2}{\epsilon}    \text{ and }
  \langle \nabla \varphi(u),\chi(u)\rangle \leq 0,   \text{ for any } u\in \mathcal{W}.
\end{equation}
Since $0<\epsilon <\frac{1}{2}$, we have $\mathcal{W}$ covers $\varphi^{d_k+\epsilon}\bigcup(X\backslash B_{R})$, this shows $\mathbf{(a)}$.
By the construction of $\chi(u)$, we know that $\chi(u)$ is  locally Lipschitz continuous and $\tau$-locally Lipschitz $\tau$-continuous on $\varphi^{ d_k+\epsilon}$,
and $\mathbf{(b)}$ is proved.

Moreover, by the choice of $\epsilon$  in (\ref{fountain.epsilon}), we have
$$\sup\limits_{\|u\|_{\tau}\leq \delta }\varphi(u)<b_{k}-\epsilon,$$
i.e.,
$$\{u\in X: \|u\|_{\tau}\leq \delta \}\subset \varphi^{b_{k}-\epsilon}.$$
So, by (\ref{fountain.chi}),
\begin{equation}\label{fountain.2.5}
   \langle \nabla\varphi(u),\chi(u)\rangle <-1, \text{ for any } u\in \varphi^{-1}[b_{k}-\epsilon,d_k+\epsilon]\bigcap B_{R},
\end{equation}
which implies $\mathbf{(c)}$.
Then, by the definition of $\chi(u)$, i.e., (\ref{fountain.chi}), and the properties of $\xi(u)$, we see that $\mathbf{(d)}$  holds. So, the claim is proved.

Next, we turn to proving $(\mathbf{i})$-$(\mathbf{v})$ of the lemma. For this purpose, we construct a map $\eta$ through the following Cauchy problem:

\begin{equation}\label{fountain.2.yita}
 \left\{
   \begin{array}{ll}
    \frac{d\eta}{dt}=\chi(\eta)\\
     \eta(0,u)=u \in \mathcal{W} .
   \end{array}
   \right.
\end{equation}

By the standard theory of ordinary differential equation in Banach space, we know that the initial problem has a
unique solution $\eta(t,u)$  on $[0,\infty)$. Furthermore, the similar argument to the proof of \cite[Lemma 6.8]{MW} yields
that $\eta$ is $\tau$-continuous. Moreover, $\eta(t,\cdot) : B_{k}\rightarrow \eta (t, B_{k})$ is a $\tau$-homeomorphism for any $t\in [0,T]$.  So, part $(\mathbf{iii})$ is proved.

Let  $B_k$ and $B_R$ be given by (\ref{fountain.1.2}) and (\ref{fountain.R}). Taking
\begin{equation}\label{fountain.2.daT}
    T=d_{k}-b_{k}+2\epsilon.
\end{equation}
Then,   $\{\eta(t,u): 0\leq t \leq T , u\in B_{k}\}\subset B_{R}$.
Indeed, it follows from (\ref{fountain.2.yita}) that
$$\eta(t,u)=u+\int_{0}^{t} \chi(\eta(s,u))  ds, \text{ for } u \in B_k.$$
By the definition of $d_k$ (see condition $(A_1)$), we know that $B_k\subset \mathcal{W}$. Then, by (\ref{fountain.epsilon}), (\ref{fountain.R}) and (\ref{fountain.2.xing}), we have  for any  $u \in B_k$  and $t\in[0,T]$
\begin{eqnarray*}
  \|\eta(t,u) \|&\leq& \|u\|+ \int_{0}^{t} \|\chi(\eta(s,u)) \| ds \\
     &\leq& \|u\|+ \int_{0}^{t} \frac{2}{\epsilon} ds
     \leq \|u\|+ \frac{2T}{\epsilon}\leq R.
\end{eqnarray*}
So, $(\mathbf{i})$ is obvious by the oddness  of $\chi(u)$.
By $(\mathbf{a})$ we have
\begin{equation*}
    \frac{d}{dt} \varphi(\eta(t,u))=\langle \nabla \varphi(u),\chi(u)\rangle \leq 0, \text{ for any } u \in B_k,
\end{equation*}
so, $\varphi(\eta(t,u))$ is non-increasing in $t\in [0,T]$ for fixed $u\in B_{k}$
and $(\mathbf{ii})$ is proved.

Now, we claim that $\eta(T,B_{k})\subset \varphi^{b_{k}-\epsilon}$. Otherwise, there
exists  $u\in B_{k}$ such that
\begin{equation}\label{fountain.2.6}
\varphi (\eta(T,u))>b_k-\epsilon.
\end{equation}
Since $\eta(t,u)$ is non-increasing along $t$,  we have

\begin{equation}\label{fountain.2.2}
     \eta(t,u)\in  \varphi^{-1}[b_{k}-\epsilon,d_k+\epsilon)\bigcap B_{R}, \text{ for any } t\in [0,T].
\end{equation}
Then, using (\ref{fountain.2.5}) we see that
\begin{eqnarray*}
  \varphi(\eta(T,u) &=& \varphi(\eta(0,u)+ \int_{0}^{T}  \langle \varphi'(\eta(s,u)), \chi(\eta(s,u)) \rangle ds\\
         &\leq& \varphi(\eta(0,u)+ \int_{0}^{T}  -1 ds \\
         &\leq& d_{k}+\epsilon- T=b_{k}-\epsilon,
\end{eqnarray*}
which   contradicts to (\ref{fountain.2.6}), and $(\mathbf{iv})$ is proved.

Finally, by $(\mathbf{d})$ and $(\mathbf{iii})$,   similar   to the proof of Lemma 6.8 of \cite{MW}, we see that
 $(\mathbf{v})$ also holds.
\end{proof}
\\

For $\tau_k$-norm   defined  in (\ref{fountain.tau k norm}),  the same as in \cite{Rup} we introduce the following definition
\begin{definition}
Let $ B_k$ and $N_k $ be defined in  (\ref{fountain.1.2}). For any $T>0$,
the mapping $\gamma : [0,T] \times B_{k}\rightarrow X$ is a $\tau_k$-admissible homotopy if
\begin{itemize}
  \item $\gamma$ is $\tau_k$-continuous in the sense that
             $$\gamma(t_{m},u_{m}) \overset{m}{\rightarrow} \gamma(t,u)  \text{ in }  \tau_k\text{-topology},
          \text{ if } t_{m}\overset{m}{\rightarrow} t \text{ and } u_{m}\overset{m}{\rightarrow} u \text{ in } \tau_k\text{-topology}.$$
  \item  For any $(t,u)\in [0,T] \times B_{k}$ there exist a neighborhood $W_{(t,u)}$ of $(t,u)$ in
          the $|\cdot|\times\tau_k$-topology such that
          $$\{  v-\gamma(s,v)| (s,v)\in  W_{(t,u)}\bigcap  ([0,T] \times B_{k}  )       \}$$
          is contained in a finite-dimensional subspace of $X$.
\end{itemize}

\end{definition}
We remark that such $\gamma$ does  exist since the  identity mapping $I_d: I_d(t,u)\equiv u$ is a $\tau_k$-admissible homotopy.

 Let $ Y_k$ and $Z_k $ be defined in  (\ref{fountain.1.1}),  $ B_k$ and $N_k $  be defined in  (\ref{fountain.1.2}).
 The following intersection lemma is proved in \cite{Rup} where the author estimated the genus of $(\gamma(t,B_{k})\bigcap N_{k} )$
(see also \cite{Ding}).

\begin{lemma}\cite[Proposition 7]{Rup} \label{fountain.intersection}
Let $\varphi \in C^{1}(X,\mathbb{R})$ be an even functional and let $\gamma : [0,T] \times B_{k}\rightarrow X$ be  a $\tau_k$-admissible homotopy  with the following properties:
\begin{itemize}
  \item $\gamma(0,u)=u$, for any $u\in B_{k}$,
  \item   $\gamma(t,-u)=-\gamma(t,u)$,
  \item $\varphi(\gamma(t,u))$ is non-increasing in $t \in [0,T]$ for fixed $u \in B_k$ ,
  \item for any $t\in [0,T]$, $\gamma(t,\cdot) : B_{k}\rightarrow \gamma (t, B_{k})$ is a
          $\tau_k$-homeomorphism.
\end{itemize}
If  $\sup\limits_{u\in Y_{k},\|u\|= \rho_{k}}\varphi(u) < \inf\limits_{u\in Z_{k},\|u\|\leq r_{k}}\varphi(u)$,
with $0< r_{k}<\rho_{k}$ , then
$$ \gamma (t, B_{k})\bigcap N_{k}\neq \emptyset \quad \text{for any}  \quad t\in [0,T],$$
where $B_k$ is given by (\ref{fountain.1.2}).

\end{lemma}

Now we come to prove Lemma \ref{fountain.guodu}.

\begin{proof}[Proof of Lemma \ref{fountain.guodu}]

By contradition, if the conclusion of Lemma \ref{fountain.guodu} is false, then, there exists $\epsilon>0$
such that
$$\|\varphi'(u)\|>\epsilon, \text{ for any } u\in E,$$
where $E$ is defined by (\ref{fountain.definition of E}). By Lemma \ref{fountain.Deformation Lemma} we know that
there exists a map $\eta(t,u) \in C([0,T] \times B_{k},X)$ satisfying  Lemma \ref{fountain.Deformation Lemma} $(\mathbf{i})$-$(\mathbf{v})$.
By Remark \ref{fountain.remark2},  the $\tau$-topology and $\tau_k$-topology are equivalent,
then it is  easy to see that $\eta(t,u)$  satisfies the   assumptions of Lemma \ref{fountain.intersection},
hence,
$$\eta(T,B_{k})\bigcap N_{k}\neq \emptyset,$$
and the definition of $b_k$ implies that
$$ \sup\limits_{u\in B_{k}}\varphi(\eta(T,u)) \geq b_{k}.$$
However,   Lemma \ref{fountain.Deformation Lemma} $(\mathbf{iv})$ shows that
$$ \sup\limits_{u\in B_{k}}\varphi(\eta(T,u)) \leq b_{k}-\epsilon,$$
which leads to a contradiction. So, the proof is complete.

\end{proof}

\begin{proof}[ Proof of Theorem \ref{fountain.Fountain Thm}]

Taking $\delta_1>0$, by $(A_4)$ we know that
$$\sup\limits_{\|u\|_{\tau}<\delta_{1}}\varphi(u)\leq C_{\delta_1},$$
for some $C_{\delta_1}>0$.
Then condition $(A_3)$ implies that there exists   $k_{1}\in\mathbf{N}$ sufficiently large such that
$$ b_{k_{1}}>\sup\limits_{\|u\|_{\tau}<\delta_{1}}\varphi(u). $$
By Lemma \ref{fountain.guodu}, we know that there exists a sequence $\{u_{n}^{k_{1}}\}$ satisfies that
$$ \varphi'(u_{n}^{k_{1}})\overset{n}{\rightarrow} 0  \text{ in } X', \quad
 \sup\limits_{n}\varphi(u_{n}^{k_{1}})<d_{k_{1}}+1
\text{ and }
  \inf\limits_{n}\|u_{n}^{k_{1}}\|_{\tau}\geq\frac{\delta_{1}}{2}.$$
Since $\varphi$ satisfies the  $(PS)^c$ condition,  $\{u_{n}^{k_{1}}\}$ has a subsequence which is
convergent to a critical point $u^{k_1}$ of $\varphi$ with $\|u^{k_1}\|\geq\|u^{k_1}\|_{\tau}\geq \frac{\delta_{1}}{2}. $

Now, we take $\delta_{2}>2 \|u^{k_1}\| $ and similar to the above there exists $k_{2} >k_1$ large enough such that
$$ b_{k_{2}}>\sup\limits_{\|u\|_{\tau}<\delta_{2}}\varphi(u),$$
and we can find the second critical point $u^{k_2}$ with $\|u^{k_2}\|\geq \frac{\delta_2}{2}>\|u^{k_1}\|$. Clearly, $u^{k_2} \neq u^{k_1}$.

Repeat the above procedures, we get a sequence of critical points $\{u^{k_m}\}$ with
\begin{center}
    $\lim\limits_{m\rightarrow\infty}\|u^{k_m}\| \rightarrow  \infty$.
\end{center}
So, the theorem is proved.
\end{proof}

\section{An application}
\noindent

The aim of this section is to apply Theorem \ref{fountain.Fountain Thm}  to prove the existence of infinitely many solutions of problem (\ref{fountain.P}).
  In this section,   $X=H^{1}(\mathbb{R}^{N})$, $N\geq3$   with the norm $\| u \|_{H^1}=(\int_{\mathbb{R}^{N}}
 (|u |^{2} + | \nabla u|^{2}) dx)^{\frac{1}{2}}$.
 $L^{p}(a(x),\mathbb{R}^{N})$ is the Lebesgue space with positive weight $a(x)$
 endowed with the norm $|u|_{L^{p}_{a(x)}}:=
 (\int_{\mathbb{R}^{N}} a(x)| u |^{p} dx)^{\frac{1}{p}}$, and this norm is simply denoted by $|u|_{L^p}$ if $a(x)\equiv 1$.
 For $r>0$, $B(x,r)=\{x \in \mathbb{R}^N : |x|<r \}$.

The variational  functional of (\ref{fountain.P}) is defined by
\begin{equation}
 \varphi(u):=\frac{1}{2}\int_{\mathbb{R}^{N}}(  | \nabla u|^{2}+V(x)| u |^{2}) dx
    - \frac{1}{q}\int_{\mathbb{R}^{N}} g(x)| u | ^{q}dx
     -\frac{1}{p}\int_{\mathbb{R}^{N}} h(x)| u | ^{p}dx,
\end{equation}
for $ u \in H^{1}(\mathbb{R}^{N})$.
By $(H_{1})$-$(H_{3})$,  $\varphi (u) \in C^1(H^{1}(\mathbb{R}^{N}))$
and
\begin{equation}
 \varphi' (u)\phi=\int_{\mathbb{R}^{N}}\Big(\nabla u\nabla\phi
+V(x)u \phi \Big)dx-\int_{\mathbb{R}^{N}}g(x)| u |^{q-2}u\phi dx
-\int_{\mathbb{R}^{N}}h(x)| u |^{p-2}u\phi dx,
\end{equation}
for any $\phi\in H^{1}(\mathbb{R}^{N})$.
Moreover,  $\varphi'$ is weakly sequentially continuous by \cite[Theorem A.2]{MW}.

Let  $L:=-\Delta+V(x)$ be the Schr$\ddot{\text{o}}$dinger operator acting on $L^{2}(\mathbb{R}^N)$ with domain $\mathcal{D}(L)=H^{2}(\mathbb{R}^N)$.
Since $L$ is self-adjoint and  $0$ lies in a gap of the spectrum of $L$,
by the standard spectral theory we know that
 the space $H^{1}(\mathbb{R}^{N})$ can be decomposed as
 $H^{1}(\mathbb{R}^{N})=Y\bigoplus Z$
such that the quadratic form:
\begin{equation}
u\in H^{1}(\mathbb{R}^{N}) \rightarrow \int_{\mathbb{R}^{N}}
 (  | \nabla u|^{2}+V(x)| u |^{2}) dx
\end{equation}
is negative and positive definite on $Y $and $ Z$ respectively, and both $Y$ and $Z$
may be infinite-dimensional.
Let $X=\mathcal{D}(|L|^{\frac{1}{2}})$ be equipped with the inner product
\begin{equation}
 \langle  u , v\rangle _{1}:=    \langle |L|^{\frac{1}{2}}u,   |L|^{\frac{1}{2}}v\rangle _{L^2}    ,
\end{equation}
and norm
\begin{equation*}
    \|u\|:=\langle |L|^{\frac{1}{2}}u,   |L|^{\frac{1}{2}}u \rangle ^{\frac{1}{2}}_{L^2},
\end{equation*}
where $\langle \cdot,\cdot \rangle_{L^2}$ is the usual inner product in $L^{2}(\mathbb{R}^{N})$.
By the condition $(H_{1})$, similar to the appendix of \cite{cos}, we know that  $X=H^{1}(\mathbb{R}^{N})$ and  the norms $\| \cdot\|$ and $\| \cdot\|_{H^1}$ are equivalent.
Moreover, $Y$ and $Z$ are also orthogonal with respect to $ \langle \cdot,\cdot\rangle_{1} $.

Let $P :X\rightarrow Y$ and $Q :X \rightarrow Z $  be the orthogonal projections,
  $(3.1)$  can be written as
\begin{equation}
 \varphi(u):=\frac{1}{2}(-\| Pu \|^{2}+\| Qu \|^{2})
    - \frac{1}{q}\int_{\mathbb{R}^{N}} g(x)| u | ^{q}dx
     -\frac{1}{p}\int_{\mathbb{R}^{N}} h(x)| u | ^{p}dx.
\end{equation}

Now,  we set
\begin{center}
   $ Y_{k}:=Y \bigoplus (\bigoplus_{j=0}^{k}\mathbb{R}f_{j})$
and
$ Z_{k}:=\overline{\bigoplus_{j=k}^{\infty}\mathbb{R}f_{j}}$,
\end{center}
where $\{f_{j}\}_{j\geq0}$ is an orthonormal basis of $(Z,\parallel\cdot\parallel)$.

Before proving Theorem \ref{fountain.existence result}, we give some useful lemmas. The first lemma is the following embedding result which has been used in many papers (see, e.g.,
\cite{Tonkes}).  Here we give a simple proof for completeness.

\begin{lemma}\label{fountain.compact embedding}
If $1<q<2^{\ast}$ and $a(x)\in L^{q_{0}}(\mathbb{R}^{N})
\bigcap L^{\infty}(\mathbb{R}^{N})$ with $a(x)\geq 0$ a.e. in $\mathbb{R}^{N}$, where $q_{0}=\frac{2N}{2N-qN+2q}$. Then,
$H^{1}(\mathbb{R}^{N})\hookrightarrow L^{q}(a(x),\mathbb{R}^{N})$ is compact.
\end{lemma}
\begin{proof}
For $u\in H^{1}(\mathbb{R}^{n})$, by the H$\ddot{\text{o}}$lder   and Sobolev inequalities, we see that
\begin{eqnarray*}
  \int_{\mathbb{R}^{N}}a(x)| u |^{q} dx &\leq& | a(x)|_{L^{q_{0}}}\cdot \left(\int_{\mathbb{R}^{N}}| u |^{\frac{2N}{N-2}} dx\right)^{\frac{qN-2q}{2N}}\\
   &=&   | a(x)|_{L^{q_{0}}}\cdot | u | ^{q}_{L^{2^{\ast}}(\mathbb{R}^{N})}  \leq   C \| u \|^{q},
\end{eqnarray*}
that is,
 $ |u|_{L^{q}_{a(x)} }\leq  C \| u \| $,
 which means that
 $H^{1}(\mathbb{R}^{N})\hookrightarrow L^{q}(a(x),\mathbb{R}^{N})$.
Let $\{u_{n}\}$ be a bounded sequence in $H^{1}(\mathbb{R}^{N})$,
passing to a subsequence, we may assume that, for some $u \in H^{1}(\mathbb{R}^{N})$,
\begin{equation*}
    u_n \overset{n}{\rightharpoonup} u \text{ weakly  in } H^{1}(\mathbb{R}^{N}) \text{ and } u_n \overset{n}{\rightarrow} u  \text{ in } L^q_{loc}(\mathbb{R}^{N}), \text{ for } 1<q<2^*.
\end{equation*}
To show that $H^{1}(\mathbb{R}^{N})\hookrightarrow L^{q}(a(x),\mathbb{R}^{N})$ is compact, we need only to show that $u_n$ strongly converges to $u$
in $L^{q}(a(x),\mathbb{R}^{N})$ for $q \in (1,2^*)$. In fact, for any $R>0$,

\begin{eqnarray*}
 && \int_{\mathbb{R}^{N}} a(x)| u_{n}-u|^{q}dx = \int_{\mathbb{R}^{N}\backslash B(0,R)} a(x)| u_{n}-u|^{q}dx
+\int_{B(0,R)} a(x)| u_{n}-u|^{q}dx \\
   & \leq& \left(\int_{\mathbb{R}^{N} \backslash B(0,R)}| a(x) |^{q_{0}} dx\right)^{\frac{1}{q_{0}}} | u_{n}-u|_{L^{2^{\ast}}(\mathbb{R}^{N})}
   + | a|_{L^{\infty}(\mathbb{R}^{N})} \int_{B(0,R)} | u_{n}-u|^{q} dx \\
   && \rightarrow 0,
\end{eqnarray*}
by letting $n \rightarrow +\infty$ and then $R \rightarrow +\infty$.
\end{proof}

\begin{lemma}
Under condition $(H_{3})$,  let
$$\beta_{k}:=\sup\limits_{u\in Z_{k},\| u \|=1} | u |_{L_{h(x)}^{p}}, \text{ for any } k \in \mathbf{N},$$
then,
 $\beta_{k}\rightarrow 0$  as    $k\rightarrow\infty$.
\end{lemma}

\begin{proof}
It is clear that $0<\beta_{k+1}\leq\beta_{k}$, and then $\beta_{k} \rightarrow\beta \geq 0$ as $k \rightarrow \infty$.
For every $k\geq0$, there exists $u_{k}\in Z_{k}$ with $\| u_{k} \|=1 $ and $| u_{k}|_{L^{p}_{h(x)}}\geq\frac{\beta_{k}}{2}$.
By the definition of $Z_{k}$,  we have $u_{k}\overset{k}{\rightharpoonup} 0$ in $H^{1}(\mathbb{R}^{N})$. Thus Lemma  \ref{fountain.compact embedding} implies that
 $u_{k}\overset{n}{\rightarrow} 0$ strongly in $L_{g(x)}^{q}$, then, $\beta=0$.
\end{proof}

\begin{lemma}\label{fountain.PSc}
If $(H_{1})-(H_{3})$ hold,
then
$\varphi$ satisfies  $(PS)^c$ condition in $H^1(\mathbb{R}^N)$ for any $c<+\infty$.
\end{lemma}

\begin{proof}
Let   $\{u_{n}\} \subset H^1(\mathbb{R}^N)$ be any sequence satisfying
  $$\sup\limits_{n}\varphi(u_{n})\leq c \text{ and } \varphi'(u_{n}) \overset{n}{\rightarrow} 0.$$
We claim that $\{u_{n}\}$ is bounded in $H^{1}(\mathbb{R}^{N})$.
Indeed, for $n$ large, there holds
\begin{eqnarray*}
  c+1+\| u_{n}\|  &\geq&\varphi(u_{n}) -  \frac{1}{2}\langle \nabla\varphi(u_{n}),u_{n}\rangle\\
   &=&  (\frac{1}{2}-\frac{1}{p})\int_{\mathbb{R}^{N}}h(x)| u_n|^{p}dx
+(\frac{1}{2}-\frac{1}{q}) \int_{\mathbb{R}^{N}}g(x)| u_n|^{q}dx,
\end{eqnarray*}
that is,
\begin{eqnarray}
 \nonumber   \int_{\mathbb{R}^{N}}h(x)| u_n|^{p}dx  &\leq& C+\| u_{n}\|+C\int_{\mathbb{R}^{N}}|g(x)|| u_n|^{q}dx.\\
            &\leq& C+\| u_{n}\|+ C\| u_{n}\|^{q}. \label{fountain.3.6}
\end{eqnarray}
Let $u_{n}=y_{n}+z_{n}$, with $y_{n} \in Y ,z_{n} \in Z $. For $n$ large,
\begin{equation*}
    \| z_{n}\| \geq \langle \varphi'(u_{n}),z_{n}\rangle
  =  \| z_{n}\|^{2}-\int_{\mathbb{R}^{N}}g(x)| u|^{q-2}u  z_{n} dx
-\int_{\mathbb{R}^{N}}h(x)| u|^{p-2}u z_{n} dx,
\end{equation*}
thus
\begin{equation*}
    \| z_{n}\|^{2}\leq \| z_{n}\|+\int_{\mathbb{R}^{N}}g(x)| u|^{q-1}  z_{n} dx
+ \int_{\mathbb{R}^{N}}h(x)| u|^{p-1} z_{n} dx.
\end{equation*}
By H$\ddot{\text{o}}$lder inequality and  Sobolev embeddings, we see that, for some $C>0$,
\begin{eqnarray*}
  \| z_{n}\|^{2} & \leq& \| z_{n}\|+
|g(x)^{\frac{q-1}{q}}u_{n}^{q-1}|_{L^{\frac{q}{q-1}}} |g(x)^{\frac{1}{q}}z_{n}^{q-1}|_{L^{\frac{q}{q-1}}}
+|h(x)^{\frac{p-1}{p}}u_{n}^{p-1}|_{L^{\frac{p}{p-1}}} |h(x)^{\frac{1}{p}}z_{n}^{p-1}|_{L^{\frac{p}{p-1}}}\\
   &=& \| z_{n}\|+
| u_{n} |^{q-1}_{L^{q}_{g(x)}} | z_{n} |_{L^{q}_{g(x)}}
+\big(\int_{\mathbb{R}^{N}}h(x)|u|^{p} dx\big)^{\frac{p-1}{p}} |z_{n}|_{L^{p}_{h(x)}}\\
   &\leq& \|z_{n}\| + C\|u_{n}\|^{q-1}\| z_{n}\|
+ C(1+\|u_{n}\|+\|u_{n}\|^{q})^{\frac{p-1}{p}} \| z_{n}\|, \text{ by } (\ref{fountain.3.6}),\\
   &\leq&\| u_{n}\| + C\| u_{n}\|^{q}
+ C(1+\|u_{n}\|+\|u_{n}\|^{q})^{\frac{p-1}{p}} \| u_{n}\|.
\end{eqnarray*}
Similarly, it follows from $\|y_{n}\| \geq -\langle \varphi'(u_{n}),y_{n}\rangle$ that
\begin{eqnarray*}
  \| y_{n}\|^{2} & \leq& \| y_{n}\|+
|g(x)^{\frac{q-1}{q}}u_{n}^{q-1}|_{L^{\frac{q}{q-1}}}|g(x)^{\frac{1}{q}}y_{n}^{q-1}|_{L^{\frac{q}{q-1}}}
+|h(x)^{\frac{p-1}{p}}u_{n}^{p-1}|_{L^{\frac{p}{p-1}}}|h(x)^{\frac{1}{p}}y_{n}^{p-1}|_{L^{\frac{p}{p-1}}}\\
   &\leq& \| u_{n}\| + C\| u_{n}\|^{q}
+ C (1+\|u_{n}\|+\|u_{n}\|^{q})^{\frac{p-1}{p}} \| u_{n}\|.
\end{eqnarray*}
Since $\|u_{n}\|^{2}=\| y_{n}\|^{2}+\| z_{n}\|^{2}$ ,
the above conclusions show that
$$  \| u_{n}\|^{2} \leq 2\| u_{n}\| + C\| u_{n}\|^{q}
+ C (1+\|u_{n}\|+\|u_{n}\|^{q})^{\frac{p-1}{p}} \| u_{n}\|.$$
Thus,  by $q<\frac{p}{p-1}$, we know $\{u_{n}\}$ is bounded in $H^{1}(\mathbb{R}^{N})$.

By the boundedness of $\{u_{n}\}$, we may assume, up to a subsequence, that
\begin{center}
    $y_{n}\overset{n}{\rightharpoonup} y$ in $H^{1}(\mathbb{R}^{N})$
\text{ and }
    $z_{n} \overset{n}{\rightharpoonup} z$ in $H^{1}(\mathbb{R}^{N}).$
\end{center}
Let $u=y+z$, we get that
\begin{center}
    $\langle \nabla\varphi(u_{n})-\nabla\varphi(u),y_{n}-y\rangle\rightarrow 0$  as $n\rightarrow\infty$.
\end{center}
\begin{eqnarray*}
  \langle \varphi'(u_{n})-\varphi'(u),y_{n}-y\rangle=-\|y_{n}-y\|^{2} &+&
\int_{\mathbb{R}^{N}}g(x)(|u|^{q-2}u-|u_{n}|^{q-2}u_{n})(y_{n}-y)dx \\
   &+& \int_{\mathbb{R}^{N}}h(x)(|u|^{p-2}u-|u_{n}|^{p-2}u_{n})(y_{n}-y)dx.
\end{eqnarray*}

Using the H$\ddot{\text{o}}$lder inequality, we see that
\begin{center}
    $y_{n} \overset{n}{\rightarrow} y$ in $H^{1}(\mathbb{R}^{N}).$
\end{center}
Similarly,
\begin{center}
    $z_{n} \overset{n}{\rightarrow} z$ in $H^{1}(\mathbb{R}^{N}).$
\end{center}
Hence,
\begin{center}
    $u_{n} \overset{n}{\rightarrow} u$ in $H^{1}(\mathbb{R}^{N}).$
\end{center}
So, we proved that $\varphi$ satisfies the  $(PS)^c$ condition for any $c<+\infty$.

\end{proof}

\begin{lemma}\label{fountain.lemma3.4}
Under the conditions $(H_{1})-(H_{3})$,
 for any $\delta>0$, there exists $C_\delta<\infty$  such that $\sup\limits_{\|u\|_{\tau}<\delta }\varphi(u) <C_\delta $.

\end{lemma}

\begin{proof}
By $u \in H^1(\mathbb{R}^N)=Y\bigoplus Z$ with $Z=Y^{\bot}$ we may set $u=y+z$ for some $y \in Y$ and $z\in Z$, then,
\begin{eqnarray*}
  \varphi(u)&=&\frac{1}{2}(-\| y \|^{2}+\| z \|^{2})
    - \frac{1}{q}\int_{\mathbb{R}^{N}} g(x)|u | ^{q}dx
     -\frac{1}{p}\int_{\mathbb{R}^{N}} h(x)| u | ^{p}dx \\
   &\leq& \frac{1}{2}(-\| y \|^{2}+\| z \|^{2})
    + \frac{1}{q}\int_{\mathbb{R}^{N}} |g(x)||u | ^{q}dx\\
   &\leq& -\frac{1}{2}\| y \|^{2}+C\|y\|^{q}+\frac{1}{2}\| z \|^{2}+C\|z\|^{q}.
\end{eqnarray*}
Since $q<2$, $-\frac{1}{2}\| y \|^{2}+C\|y\|^{q}$ is bounded from above. By (\ref{fountain.tau norm}), we have  $\|z\|\leq\|u\|_{\tau}\leq \delta$, so,
there exists $C_\delta<\infty$  such that
$$ \sup\limits_{\|u\|_{\tau}<\delta }\varphi(u) <C_\delta.$$

\end{proof}

\begin{proof}[Proof of Theorem \ref{fountain.existence result}]
By Theorem \ref{fountain.Fountain Thm} and Lemmas \ref{fountain.PSc}-\ref{fountain.lemma3.4},
in order to prove Theorem \ref{fountain.existence result} we need only to verify the conditions $(A_{1})$, $(A_{2})'$ and $(A_{3})$.

Clearly, $(A_1)$ is true since $\varphi$ maps a bounded set into a bounded set.

Next, we prove $(A_{2})'$  by showing that $a_{k}\rightarrow - \infty $ as $\rho_{k} \rightarrow  \infty $.
Let $u=y+z$ with $y \in Y$ and $z \in Z $. For any $u\in Y_{k}$,
by Lemma \ref{fountain.compact embedding}, we  see that
\begin{eqnarray*}
  \varphi(u)&=&\frac{1}{2}(-\| y \|^{2}+\| z \|^{2})
    - \frac{1}{q}\int_{\mathbb{R}^{N}} g(x)|u | ^{q}dx
     -\frac{1}{p}\int_{\mathbb{R}^{N}} h(x)| u | ^{p}dx \\
   &\leq& \frac{1}{2}(-\| y \|^{2}+\| z \|^{2})
    +\frac{1}{q}| u|^{q}_{L^{q}_{|g(x)|}}
     -\frac{1}{p}  | u|^{p}_{L^{p}_{h(x)}}\\
   &\leq& \frac{1}{2}(-\| y \|^{2}+\| z \|^{2})
   +\frac{2^{q-1}}{q}(\|y\|^{q}_{L^{q}_{|g(x)|}}+\|z\|^{q}_{L^{q}_{|g(x)|}})
    - \frac{1}{p}| u|^{p}_{L^{p}_{h(x)}}\\
    &\leq&  \frac{1}{2}(-\| y \|^{2}+\| z \|^{2}) +C_q (\|y\|^{q}+\|z\|^{q})- \frac{1}{p}| u|^{p}_{L^{p}_{h(x)}}.
\end{eqnarray*}

Since   $H^{1}(\mathbb{R}^{N}) \hookrightarrow L^{p}(h(x), \mathbb{R}^{N})$, we denote by
$E_{k}$ the closure of $Y_{k}$ in $L^{p}(h(x),\mathbb{R}^{N})$, then there exists a continuous projection from $E_{k}$
to $\bigoplus_{j=0}^{k}f_{j}$, thus there exists a constant $C>0$ such that
$$| z|^{p}_{L^{p}_{h(x)}} \leq C| u|^{p}_{L^{p}_{h(x)}},$$
and note that all norms are equivalent in a finite-dimensional vector space, then for any $z \in \bigoplus_{j=0}^{k}f_{j}$, there exists $C>0$ such that
$$\| z \|^{p}\leq C | z|^{p}_{L^{p}_{h(x)}},$$
thus
$$
  \varphi(u) \leq (-\frac{1}{2}\| y \|^{2}+C\|y\|^{q})+(\frac{1}{2}\| z \|^{2}
    +C\|z\|^{q}-C\| z \|^{p}).
$$
So,
$$a_{k}:=\sup\limits_{u\in Y_{k},\|u\|= \rho_{k}}\varphi(u)\rightarrow -\infty \quad \text{as} \quad  \rho_{k}\rightarrow \infty.$$

Finally, for any $u\in Z_{k}$ with $\|u\|=r_k$, let $u=y+z$ with $y \in Y$ and $z \in Z $, then $y=0$, $z=u$. Furthermore,
\begin{eqnarray*}
  \varphi(u) &=& \frac{1}{2}\| u \|^{2}- \frac{1}{q}\int_{\mathbb{R}^{N}} g(x)|u | ^{q}dx
     -\frac{1}{p}\int_{\mathbb{R}^{N}} h(x)| u | ^{p}dx \\
   &\geq& \frac{1}{2}\| u \|^{2}- \frac{1}{q}\int_{\mathbb{R}^{N}} |g(x)||u | ^{q}dx
     -\frac{1}{p}\int_{\mathbb{R}^{N}} h(x)| u | ^{p}dx \\
   &\geq& (\frac{1}{4}\|u\|^{2}-C\|u\|^{q})+(\frac{1}{4}\|u\|^{2}-\frac{1}{p}\beta_{k}^{p}\|u\|^{p}).
\end{eqnarray*}
Choose $r_{k}=(\frac{p}{4})^{\frac{1}{p-2}}\frac{1}{\beta_{k}^{\frac{p}{p-2}}}$, we have
$$ \varphi(u) \geq \frac{1}{4}\|u\|^{2}-C\|u\|^{q}=\frac{1}{4}|r_k|^{2}-C|r_k|^{q}. $$
Since we have $r_k \overset{k}{\rightarrow} +\infty$ by
$$\beta_{k} \rightarrow 0    \quad \text{as} \quad k\rightarrow\infty,$$
then,
$$b_{k}:=\inf\limits_{u\in Z_{k},\|u\|= r_{k}}\varphi(u) \rightarrow +\infty   \quad \text{as} \quad k \rightarrow \infty.$$
Thus   $(A_{3})$ is also proved.

{\bf Acknowledgements:} This work was supported by the NSFC under grants 11471331 and 11501555.

\end{proof}


\begin{thebibliography}{99}


\bibitem{G} A. Ambrosetti, H. Brezis and G. Cerami, Combined effects of concave and convex nonlinearities
in some elliptic problems, J. Funct. Anal. 122(1994), 519-543.



\bibitem{old fountain} T. Bartsch, Infinitely many solutions of a symmetric Dirichlet problem, Nonlinear Anal. 20(1993), 1205-1216.


\bibitem{He} T. Bartsch, M. Willem, On an elliptic equation with concave and convex nonlinearities,
Proc. Amer. Math. Soc. 123(1995), 3555-3561.


\bibitem{Bartsch3} T. Bartsch, M. Willem, Infinitely many nonradial solutions of a Euclidean scalar field equation, J. Funct. Anal.  117(1993), 447-460.


\bibitem{BC1} C. J. Batkam, F. Colin, Generalized Fountain theorem and applications
to strongly indefinite semilinear problems, J. Math. Anal. Appl. 405(2013), 438-452.

\bibitem{BC2} C. J. Batkam, F. Colin, On multiple solutions of a semilinear  Schr$\ddot{\text{o}}$dinger  equation with
periodic potential, Nonlinear Anal. 84(2013), 39-49.



\bibitem{CH} S. W. Chen, C. L. Wang, An infite-dimensional linking theorem without upper semi-continuous
assumption and its applications, J. Math. Anal. Appl. 420(2014), 1552-1567.





\bibitem{cos} D. G. Costa, H. Tehrani, Existence and multiplicity results for a class of Schr$\ddot{\text{o}}$dinger equations with indefinite nonlinearities, Adv. Differential Equations 8(2003), 1319-1340.


\bibitem{cos3}   D. G. Costa, H. Tehrani, Existence of positive solutions for a class of indefinite elliptic problems in $\mathbb{R}^N$, Calc. Var. Partial Differential Equations 13(2001), 159-189.

\bibitem{cos2}   D. G. Costa, H. Tehrani, M. Ramos, Non-zero solutions for a Schr$\ddot{\text{o}}$inger equation with indefinite linear and nonlinear terms, Proc. R. Soc. Edinb. A 134(2004), 249-258.


\bibitem{Ding} Y. H. Ding, $Variational\text{ } Methods\text{ }  for\text{ }  Strongly \text{ } Indefinite \text{ } Problems$, World Scientific(2007).

\bibitem{hz} X.M. He, W.M. Zou, Multiplicity of solutions for a class of elliptic boundary value problems, Nonlinear Anal.  71(2009), 2606-2613.

\bibitem{SK} W. Kryszewski,  A. Szulkin, Generalized linking theorem with an application to semilinear Schr$\ddot{\text{o}}$dinger equation, Adv. Differential Equations 3(1998), 441-472.


 \bibitem{J F Yang} F. Liu, J. F. Yang, Nontrivial solutions of Schr$\ddot{\text{o}}$inger equations with indefinite nonlinearities, J. Math. Anal. Appl. 334(2007), 627-645.

\bibitem{SI1} S. B. Liu, Z. P. Shen, Generalized saddle point theorem and asymptotically linear problems with periodic potential, Nonlinear Anal. 86(2013), 52-57.

\bibitem{Rup} H. J. Ruppen, A generalized min-max theorem for functionals of strongly indefinite sign, Calc. Var. Partial Differential Equations 50(2014), 231-255.




\bibitem{weak linking} M. Schechter, W. M. Zou, Weak linking theorems and Schr$\ddot{\text{o}}$inger equations with critical Sobolev exponent, ESAIM  Control Optim. Calc. Var. 9(2003),  601-619.



\bibitem{Tonkes} E. Tonkes, A semilinear elliptic equation with convex and concave nonlinearities, Topol. Methods Nonlinear Anal. 13(1999), 251-271.

\bibitem{F Z Wang} F. Z. Wang,  Multiple solutions for some nonlinear Schr$\ddot{\text{o}}$inger equations with indefinite linear part, J. Math. Anal. Appl. 331(2007),  1001-1022.

\bibitem{wx} Z.Q. Wang, J.K. Xia, Ground States for Nonlinear Schrodinger Equations with a Sign-changing Potential Well,
Adv. Nonlinear Studies 15(2015), 749-762.

\bibitem{MW} M. Willem, $Minimax \text{ }  Theorems$, Birkhauser, Boston(1996).


\bibitem{ZouWM} W. M.  Zou, Variant fountain theorems and their applications, Manuscripta Math. 104(2001), 343-358.







\end{thebibliography}
\end{document}